\documentclass{mcom-l}

\usepackage{amsmath,arydshln,multirow}

\usepackage{arydshln}
\usepackage{cases}
\usepackage{amsmath}
\usepackage{amsfonts}
\usepackage{bm}
\usepackage{arydshln}
\usepackage{amsfonts,amsmath,amssymb,amscd,bbm,amsthm,mathrsfs,dsfont}
\usepackage{mathrsfs}
\usepackage{pb-diagram}
\usepackage{amssymb}
\usepackage{xypic}
\usepackage[dvips]{graphicx}
\usepackage[all]{xy}
\usepackage{CJK}
\usepackage[dvipsnames]{xcolor}
\usepackage{mathtools}

\newtheorem{theorem}{Theorem}[section]

\newtheorem{proposition}[theorem]{Proposition}
\newtheorem{lemma}[theorem]{Lemma}

\theoremstyle{definition}
\newtheorem{definition}[theorem]{Definition}

\newtheorem{proposition-definition}[theorem]{Proposition-Definition}
\newtheorem{corollary}[theorem]{Corollary}

\theoremstyle{remark}
\newtheorem{remark}[theorem]{Remark}

\numberwithin{equation}{section}

\begin{document}

% \title[short text for running head]{full title}
\title[On some combinatorial properties of generalized cluster algebras]{On some combinatorial properties of \\generalized cluster algebras}

%    Only \author and \address are required; other information is
%    optional.  Remove any unused author tags.

%    author one information
% \author[short version for running head]{name for top of paper}
\author{Peigen Cao}
\address{Department
of Mathematics, Zhejiang University (Yuquan Campus), Hangzhou, Zhejiang
310027,  P.R.China}
\curraddr{}
\email{peigencao@126.com}
\thanks{}

%    author two information
\author{Fang Li}
\address{Department
of Mathematics, Zhejiang University (Yuquan Campus), Hangzhou, Zhejiang
310027,  P.R.China}
\curraddr{}
\email{fangli@zju.edu.cn}
\thanks{}

%    \subjclass is required.

\date{}

\dedicatory{}


%    Abstract is required.
\begin{abstract}
In this paper, we prove some combinatorial results on generalized cluster algebras. To be more precisely, we prove that (i) the seeds of a generalized cluster algebra $\mathcal A(\mathcal S)$ whose clusters contain particular cluster variables form a connected subgraph of the exchange graph of $\mathcal A(\mathcal S)$; (ii) there exists a bijection from the set of cluster variables of a generalized cluster algebra to the set of cluster variables of another generalized cluster algebra, if their initial exchange matrices satisfying a mild condition. Moreover, this bijection preserves the set of clusters of these two generalized cluster algebras.

As applications of the second result, we prove some properties of the components of the $d$-vectors of a generalized cluster algebra and we give a characterization for the clusters of a generalized cluster algebra.
\end{abstract}

\maketitle

%    Text of article.

\section{Introduction}

Cluster algebras were introduced by Fomin and Zelevinsky in \cite{FZ}. The motivation was to create a common framework for the phenomena occurring in connection
with total positivity and canonical bases. Since then, numerous connections between cluster algebras and other branches of mathematics have been discovered, for example, Poisson geometry, discrete dynamical systems, higher Teichm\"uller spaces, representation theory of quivers and finite-dimensional algebras.

Cluster algebras are commutative algebras whose generators and relations are constructed in a recursive manner. The generators of a cluster algebra are called  {\em cluster variables}, which are grouped into overlapping {\em clusters} of the same size. One remarkable feature of cluster algebras is that they have the Laurent phenomenon, which says that for any given cluster ${\bf x}_{t_0}=\{x_{1;t_0},\cdots,x_{n;t_0}\}$, any cluster variable $x_{i;t}$ can be written as a Laurent polynomial in $x_{1;t_0},\cdots,x_{n;t_0}$.

Generalized cluster algebras were introduced in \cite{CS} by Chekhov and Shapiro,
which are the generalization of the classic cluster algebras introduced by Fomin and Zelevinsky
in \cite{FZ}. In the classic case,  a product of cluster variables, one known and one
unknown, is equal to a binomial in other known variables.  These binomial exchange relations are replaced by polynomial exchange relations in generalized cluster algebras.

The generalized cluster structures naturally appear in the Teichm\"{u}ller spaces of Riemann surfaces with orbifold points \cite{CS}, WKB analysis \cite{IN}, representations of quantum affine algebras \cite{G}, Drinfeld double of $GL_n$ \cite{GSV1}. The generalized cluster algebras share many common properties with the classic cluster algebras, for example, the Laurent phenomenon, finite type classification \cite{CS}, tropical dualities phenomenon between $C$-matrices and $G$-matrices \cite{NT}, the existence of greedy bases in rank $2$ case \cite{RD}. One can also refer to \cite{NR,CL}.

In this paper we will provide some other similarity between generalized cluster algebras and classic cluster algebras.  Now we introduce our main results in this paper.
Firstly, we prove that  the seeds of a generalized cluster algebra $\mathcal A(\mathcal S)$ whose clusters contain particular cluster variables form a connected subgraph of the exchange graph of $\mathcal A(\mathcal S)$ (see Theorem \ref{thmgraph}).

Secondly, we prove that there exists a bijection from the set of cluster variables of a generalized cluster algebra to the set of cluster variables of another generalized cluster algebra, if their initial exchange matrices satisfying a mild condition. Moreover, this bijection preserves the set of clusters of these two generalized cluster algebras (see Theorem \ref{promain} for details).

As applications of Theorem \ref{promain}, we prove some properties of the components of the $d$-vectors of a generalized cluster algebra (see Theorem \ref{thmdvec}) and we give a characterization for the clusters of a generalized cluster algebra (see Theorem \ref{thmlast}).

Note that Theorem \ref{thmgraph}, Theorem \ref{thmdvec}, Theorem \ref{thmlast} for classic cluster algebras have been given in \cite{CL1}. These results provide new similarity between generalized cluster algebras and classic cluster algebras.

This paper is organized as follows. In Section 2, some basic definitions, notations and known results are introduced. In Section 3, we give the proof of Theorem Theorem \ref{thmgraph} and Theorem \ref{promain}. In Section 4, we give the applications of Theorem \ref{promain}.

\section{Preliminaries}
\subsection{Generalized cluster algebras}
Recall that $(\mathbb P, \oplus, \cdot)$ is a {\bf semifield} if $(\mathbb P,  \cdot)$ is an abelian multiplicative group endowed with a binary operation of auxiliary addition $\oplus$ which is commutative, associative and satisfies that the multiplication  distributes over the auxiliary addition.

The {\bf tropical semifield} $\mathbb P=Trop(y_1,\cdots,y_m)$ is the free (multiplicative) abelian group generated by $y_1,\cdots,y_m$
with auxiliary addition $\oplus$ defined by
$$\prod\limits_{i}y_i^{a_i}\oplus\prod\limits_{i}y_i^{b_i}=\prod\limits_{i}y_i^{min(a_i,b_i)},$$

The multiplicative group of any semifield $\mathbb P$ is torsion-free \cite{FZ}, hence its group ring $\mathbb Z\mathbb P$ is a domain.
We take an ambient field $\mathcal F$  to be the field of rational functions in $n$ independent variables with coefficients in $\mathbb Z\mathbb P$.

An integer matrix $B_{n\times n}=(b_{ij})$ is  called  {\bf skew-symmetrizable} if there is a positive integer diagonal matrix $S$ such that $SB$ is skew-symmetric, where $S$ is said to be a {\bf skew-symmetrizer} of $B$.

\begin{definition}[Seed and mutation pair]
(i)  A {\bf (labeled) seed}  in $\mathcal F$ is a triple $(B,{\bf x},{\bf y})$, where
\begin{itemize}
\item $B=(b_{ij})$ is an $n\times n$ integer skew-symmetrizable matrix, called an {\bf exchange matrix};

\item  ${\bf x}=(x_1,\dots, x_n)$ is an $n$-tuple such that $X=\{x_1,\dots, x_n\}$ is a free generating set of $\mathcal F$ over $\mathbb{ZP}$. We call ${\bf x}$  the {\bf cluster}  and $x_1,\dots,x_n$ the {\bf cluster variables}  of  $({\bf x},{\bf y},B)$;

\item ${\bf y}=(y_1,\cdots,y_n)$ is an $n$-tuple of elements in $\mathbb P$, where $y_1,\cdots,y_n$ are called {\bf coefficients}.
\end{itemize}

(ii) An {\bf (labeled) mutation pair}  in $\mathcal F$ is pair $(R,{\bf z})$, where
\begin{itemize}
\item $R=diag(r_1,\cdots,r_n)$ is a diagonal integer matrix with $r_i>0$, called a {\bf mutation degree matrix};
\item
${\bf z}=(z_{i,s})_{i=1,\cdots,n;~s=1,\cdots,r_i-1}$ a family of elements in $\mathbb P$ satisfying the reciprocity condition
$$z_{i,s}=z_{i,r_i-s}$$ for $s=1,\cdots,r_i-1$, which are called {\bf frozen coefficients}. In addition, we denote $$z_{i,0}=z_{i,r_i}=1$$ for $i=1,\cdots,n$.
\end{itemize}
\end{definition}
 Each mutation pair $(R, {\bf z})$ naturally corresponds to a collection of polynomials ${\bf Z}=(Z_1,\cdots,Z_n)$, where
 $$Z_i(u)=z_{i,0}+z_{i;1}u+\cdots+z_{i,r_i-1}u^{r_i-1}+z_{i;r_i}u^{r_i}\in\mathbb {ZP}[u].$$
We call $Z_1,\cdots,Z_n$ the {\bf mutation polynomials} of $(R, {\bf z})$.

\begin{definition}[$(R,{\bf z})$-seed mutation]
Let $(R, {\bf z})$ be a mutation pair, and ${\bf Z}=(Z_1,\cdots,Z_n)$ be the collection of mutation polynomials of $(R, {\bf z})$. Let $(B,{\bf x},{\bf y})$ be a seed, we define the {\bf $(R,{\bf z})$-seed mutation} at $k\in\{1,\cdots,n\}$ by $\mu_k(B,{\bf x},{\bf y})=(B^\prime, {\bf x}^\prime,{\bf y}^\prime)$, where

\begin{eqnarray}
b_{ij}^\prime&=&\begin{cases} -b_{ij},&\text{if }i=k\text{ or } j=k;\\b_{ij}+r_k(b_{ik}[-b_{kj}]_++[b_{ik}]_+b_{kj}),&\text{otherwise}.\end{cases}\nonumber\\
x_i^\prime&=&\begin{cases}x_i,&\text{if } i\neq k;\\ x_k^{-1}\left(\prod\limits_{j=1}^nx_j^{[-b_{jk}]_+}\right)^{r_k}\frac{Z_k(\hat y_k)}{Z_k|_{\mathbb P}(y_k)},&\text{if }i=k,\end{cases} \;\;\;\;\;\;\text{where }\hat y_k=y_k\prod\limits_{i=1}^nx_i^{b_{ik}}.\nonumber\\
y_i^\prime&=&\begin{cases} y_k^{-1}~,&\text{if } i=k;\\ y_i\left(y_k^{[b_{ki}]_+}\right)^{r_k}\left(Z_k|_{\mathbb P}(y_k)\right)^{-b_{ki}}~,&\text{if } i\neq k.\nonumber
\end{cases}
\end{eqnarray}
When the mutation degree matrix $R$  is given, we also denote $B^\prime=\mu_k(B)$, which is called the {\bf $R$-mutation} at $k$.
\end{definition}
It can be seen that $\mu_k(B,{\bf x},{\bf y})=(B^\prime, {\bf x}^\prime,{\bf y}^\prime)$ is also a seed.
\begin{remark}
(i) If $(R,{\bf z})=(I_n,\phi)$, the $(R,{\bf z})$-seed mutations are just  the {\bf (classic) seed mutations} by Fomin and Zelevinsky. We will use $\mu_k^\circ$ to denote the (classic) seed mutations.

(ii) On the level of matrix mutations, we can see $$\mu_k(B)R=\mu_k^\circ(BR),$$ where $\mu_k$ is $R$-mutation, and $\mu_k^\circ$ is $I_n$-mutation.
\end{remark}
\begin{proposition}\cite{NR} The $(R,{\bf z})$-seed mutation $\mu_k$ is an involution.
\end{proposition}
Let $\mathbb T_n$ be the  $n$-regular tree, and label the edges of $\mathbb T_n$ by  $1,\dots,n$ such that the $n$  different edges adjacent to the same vertex of $\mathbb T_n$ receive different labels.

\begin{definition}
  (i) An {\bf $(R,{\bf z})$-cluster pattern} ${\mathcal S}$ is an assignment of a seed  $(B_t,{\bf x}_t,{\bf y}_t)$ to every vertex $t$ of the $n$-regular tree $\mathbb T_n$ such that $$(B_t^\prime,{\bf x}_{t}^\prime,{\bf y}_t^\prime)=\mu_{k}(B_t,{\bf x}_t,{\bf y}_t)$$ for any edge $t^{~\underline{\quad k \quad}}~ t^{\prime}$, where $\mu_k$ is the $(R,{\bf z})$-seed mutation at $k$.

  (ii) Let $\mathcal S$ be an $(R,{\bf z})$-cluster pattern, the {\bf $(R,{\bf z})$-cluster algebra} $\mathcal A(\mathcal S)$  (also known as {\bf generalized cluster algebra}) associated with $\mathcal S$ is the $\mathbb {ZP}$-subalgebra of $\mathcal F$ generated by all the cluster variables of $\mathcal S$.
\end{definition}
\begin{remark}
If $(R,{\bf z})=(I_n,\phi)$, the $(R,{\bf z})$-cluster algebras are just the {\bf (classic) cluster algebras} by Fomin and Zelevinsky.
\end{remark}
For the seed $(B_t,{\bf x}_t,{\bf y}_t)$, we always write
$$B_t=(b_{ij}^t),\;\;{\bf x}_t=(x_{1;t},\cdots,x_{n;t}),\;\;{\bf y}_t=(y_{1;t},\cdots,y_{n;t}).$$

\begin{theorem}\cite{CS}
Let $\mathcal A(\mathcal S)$ be an $(R,{\bf z})$-cluster algebra and $(B_{t_0},{\bf x}_{t_0},{\bf y}_{t_0})$ be a seed of $\mathcal A(\mathcal S)$, then any cluster variable $x_{i;t}$ can be written as a Laurent polynomial in $\mathbb {ZP}[x_{1;t_0}^{\pm 1},\cdots,x_{n;t_0}^{\pm 1}]$.
\end{theorem}

Let $\mathcal A(\mathcal S)$ be an $(R,{\bf z})$-cluster algebra, ${\bf x}_t$ and ${\bf x}_{t_0}$ be two clusters of $\mathcal A(\mathcal S)$. We know that each $x_{i;t}$ can be viewed as a rational functions in $x_{1;t_0},\cdots,x_{n;t_0}$ with coefficients in $\mathbb {ZP}$, so we can define the corresponding Jacobi matrix as follows.
$$J^t_{t_0}=\begin{pmatrix} \frac{\partial x_{1;t}}{\partial x_{1;t_0}}&\frac{\partial x_{2;t}}{\partial x_{1;t_0}}&\cdots &\frac{\partial x_{n;t}}{\partial x_{1;t_0}}\\ \frac{\partial x_{1;t}}{\partial x_{2;t_0}}&\frac{\partial x_{2;t}}{\partial x_{2;t_0}}&\cdots &\frac{\partial x_{n;t}}{\partial x_{2;t_0}}\\ \vdots &\vdots& &\vdots\\ \frac{\partial x_{1;t}}{\partial x_{n;t_0}}&\frac{\partial x_{2;t}}{\partial x_{n;t_0}}&\cdots &\frac{\partial x_{n;t}}{\partial x_{n;t_0}} \end{pmatrix}.$$
Let $H^t_{t_0}=diag(x_{1;t_0},\cdots, x_{n;t_0})J^t_{t_0}diag(x_{1;t}^{-1}, \cdots, x_{n;t}^{-1})$, which is called the {\bf $H$-matrix} of ${\bf x}_t$ with respect to ${\bf x}_{t_0}$.

\begin{theorem}[Cluster formula \cite{CL}]
Let $\mathcal A(\mathcal S)$ be an $(R,{\bf z})$-cluster algebra,  $(B_t,{\bf x}_t,{\bf y}_t)$ and $(B_{t_0}$,${\bf x}_{t_0},{\bf y}_{t_0})$ be two seeds of $\mathcal A(\mathcal S)$. Then we have
 $$H_{t_0}^{t}(B_tR^{-1}S^{-1}) (H_{t_0}^{t})^{\rm T}=B_{t_0}R^{-1}S^{-1}\;\;\;\;\;\;\; \text{and}\;\;\;\;\;\;\; det(H_{t_0}^{t})=\pm 1,$$
where $S$ is a skew-symmetrizer of $RB_{t_0}$.
\end{theorem}

\subsection{$D$-matrices and exchange graph}
By the Laurent phenomenon, each cluster variable can be written as
\begin{eqnarray}
x_{i;t}=\frac{f(x_{1;t_0},\cdots,x_{n,t_0})}{x_{1;t_0}^{d_1}\cdots x_{n;t_0}^{d_n}},\nonumber
\end{eqnarray}
 where $f$ is a polynomial in $x_{1;t_0},\cdots,x_{n;t_0}$ with coefficients in $\mathbb {ZP}$ with $x_{j;t_0}\nmid f$ for any $j=1,\cdots,n$.
The vector ${\bf d}_{i;t}^{t_0}=(d_1,\cdots,d_n)^{\rm T}\in\mathbb Z^n$ is called the {\bf $d$-vector} of $x_{i;t}$ with respect to ${\bf x}_{t_0}$. The matrix
$D_t^{t_0}=({\bf d}_{1;t}^{t_0},\cdots,{\bf d}_{n;t}^{t_0})$ is called the {\bf $D$-matrix} of ${\bf x}_t$ with respect to ${\bf x}_{t_0}$.

\begin{proposition}\cite{CL}\label{prodmut}
Let $\mathcal A(\mathcal S)$ be an  $(R,{\bf z})$-cluster algebra with  initial seed at $t_0$.  Then the $D$-matrix $D_t^{t_0}=({\bf d}_{1;t}^{t_0},\cdots,{\bf d}_{n;t}^{t_0})$ is uniquely determined by the initial condition $D_{t_0}^{t_0}=-I_n$, together with the following relation:
\begin{equation}\label{mutationD}
{\bf d}_{j;t^\prime}^{t_0}=\begin{cases}{\bf d}_{j;t}^{t_0}  & \text{if } j\neq k;\\ -{\bf d}_{k;t}^{t_0}+max\{\sum\limits_{b_{lk}^t>0}{\bf d}_{l;t}^{t_0} b_{lk}^tr_k, \sum\limits_{b_{lk}^t<0} -{\bf d}_{l;t}^{t_0}b_{lk}^tr_k\}  &\text{if } j=k.\end{cases}\nonumber
 \end{equation}
 for any $t,t'\in\mathbb T_n$ with edge $t^{~\underline{\quad k \quad}} ~t^{\prime}$.
\end{proposition}

The following result is a direct corollary of Proposition \ref{prodmut}.
\begin{corollary}\label{cordmat}
Let $\mathcal A(\mathcal S)$ be an  $(R,{\bf z})$-cluster algebra with  initial seed  $(B_{t_0},{\bf x}_{t_0},{\bf y}_{t_0})$, and $\mathcal A(\overline{\mathcal S})$ be an  $(\overline R,\overline{\bf z})$-cluster algebra with  initial seed  $(\overline B_{t_0},\overline{\bf x}_{t_0},\overline{\bf y}_{t_0})$. If $B_{t_0}R=\overline B_{t_0}\overline R$, then for any two vertices $w,v\in\mathbb T_n$, we have
$$D_v^w=\overline D_v^w,$$
where $D_v^w$ is the $D$-matrix  of  $\mathcal A(\mathcal S)$ and $\overline D_v^w$ is the $D$-matrix  of  $\mathcal A(\overline{\mathcal S})$.
\end{corollary}

Let $\mathcal A(\mathcal S)$ be an  $(R,{\bf z})$-cluster algebra with initial seed at $t_0$. Let $(B_{t_1},{\bf x}_{t_1},{\bf y}_{t_1})$ and $(B_{t_2},{\bf x}_{t_2},{\bf y}_{t_2})$ be two seeds of $\mathcal A(\mathcal S)$. We say that the two seeds $(B_{t_1},{\bf x}_{t_1},{\bf y}_{t_1})$ and $(B_{t_2},{\bf x}_{t_2},{\bf y}_{t_2})$ are {\bf equivalent} if there exists a permutation $\sigma$ of $\{1,\cdots,n\}$ such that $$x_{i;t_2}=x_{\sigma(i);t_1},\;y_{i;t_2}=y_{\sigma(i);t_1},\;b_{ij}^{t_2}=b_{\sigma(i)\sigma(j)}^{t_1},$$
for any $i,j=1,\cdots,n$.

\begin{definition}
Let $\mathcal A(\mathcal S)$ be an  $(R,{\bf z})$-cluster algebra, the {\bf  exchange graph} ${\bf EG}(\mathcal A(\mathcal S))$ of $\mathcal A(\mathcal S)$  is  a graph satisfying that
\begin{itemize}
\item the set of vertices of ${\bf EG}(\mathcal A(\mathcal S))$ is in bijection with the set of  seeds (up to equivalence) of $\mathcal A(\mathcal S)$;
    \item two vertices joined by an edge if and only if the corresponding two seeds (up to equivalence) are obtained from each other by once mutation.
\end{itemize}
\end{definition}

\subsection{Generalized cluster algebras with principal coefficients}
Now we give the definition of principal coefficients $(R,{\bf z})$-cluster algebra.
Let us temporarily regard ${\bf y}=(y_1,\cdots,y_n)$, and ${\bf z}=(z_{i,s})_{i=1,\cdots,n;~s=1,\cdots,r_i-1}$
 with $z_{i,s}=z_{i,r_i-s}$ as formal variables. Let $\mathbb P_{pr}:=Trop({\bf y},{\bf z})$
be the tropical semifield of ${\bf y}$ and ${\bf z}$, and $\mathcal F_{pr}$  be the field of rational functions in $n$ independent variables with coefficients in $\mathbb Z\mathbb P_{pr}$.

\begin{definition}
 An $(R,{\bf z})$-cluster algebra  $\mathcal A(\mathcal S)$ in $\mathcal F_{pr}$ is said to be with {\bf principal coefficients} at $t_0$, if ${\bf y}_{t_0}={\bf y}$.
\end{definition}

Let $\mathcal A(\mathcal S)$ be an $(R,{\bf z})$-cluster algebra with principal coefficients at $t_0$. By the Laurent phenomenon, each cluster variable $x_{i;t}$ can be expressed as $$X_{i;t}({\bf x}_{t_0},{\bf y},{\bf z})=\mathbb {ZP}_{pr}[{\bf x}_{t_0}^{\pm1}]=\mathbb Z[{\bf x}_{t_0}^{\pm1},{\bf y}^{\pm1},{\bf z}^{\pm 1}].$$ We call $X_{i;t}$ the {\bf $X$-function} of $x_{i;t}$.

\begin{proposition}\cite{NT}
Each $X$-function $X_{i;t}$ is a Laurent polynomial in $\mathbb Z[{\bf x}_{t_0}^{\pm1},{\bf y},{\bf z}]$.
\end{proposition}
The {\bf $F$-polynomial} $F_{i;t}$ of $x_{i;t}$ is defined by
$F_{i;t}=X_{i;t}|_{x_{1;t_0}=\cdots=x_{n;t_0}=1}\in\mathbb Z[{\bf y},{\bf z}]$.

Let $\mathcal A(\mathcal S)$ be an $(R,{\bf z})$-cluster algebra with principal coefficients at $t_0$, we introduce a $\mathbb Z^n$-grading on $\mathbb Z[{\bf x}_{t_0}^{\pm1},{\bf y},{\bf z}]$ as follows:
$$deg(x_{i;t_0})={\bf e}_i,~~deg(y_i)=-{\bf b}_i,~~deg(z_{i;s})=0,$$
where  ${\bf e}_i$ is the $i$-th column vector of $I_n$, and ${\bf b}_i$ is the $i$-th column vector of $B_{t_0}$.

\begin{proposition}\cite{NT}
Each $X$-function $X_{i;t}$ is homogeneous with respect to the $\mathbb Z^n$-grading on $\mathbb Z[{\bf x}_{t_0}^{\pm1},{\bf y},{\bf z}]$.
\end{proposition}
Keep the above notations. The vector $g(x_{i;t}):=deg(X_{i;t})\in\mathbb Z^n$ is called the {\bf $g$-vector} of $x_{i;t}$ and the matrix
$$G_t=(g(x_{1;t}),\cdots,g(x_{n;t}))$$
is called the {\bf $G$-matrix} of ${\bf x}_t$.

\begin{proposition-definition}\cite{NT}
Let $\mathcal A(\mathcal S)$ be an $(R,{\bf z})$-cluster algebra with principal coefficients at $t_0$. Then each $y_{i;t}$ is a Laurent monomial of ${\bf y}$ with coefficient $1$, namely, $y_{i;t}$ has the form of
$$y_{i;t}=\prod\limits_{j=1}^ny_j^{c_{ji}^t}.$$
The resulting vector ${\bf c}_{i;t}=(c_{1i}^t,\cdots,c_{ni}^t)^{\rm T}$ is called a {\bf $c$-vector} and the matrix $C_t=({\bf c}_{1;t},\cdots,{\bf c}_{n;t})$ is called a {\bf $C$-matrix}.
\end{proposition-definition}

\begin{proposition}\cite{NT,CL}\label{procg}
Let $\mathcal A(\mathcal S)$ be an $(R,{\bf z})$-cluster algebra with principal coefficients at $t_0$, and $S$ be a skew-symmetrizer of $RB_{t_0}$, then
$$SRC_tR^{-1}S^{-1}G_t^{\rm T}=I_n.$$
\end{proposition}

\begin{theorem}\label{thmnt}
\cite[Theorem 3.22 and 3.23]{NT} Let $\mathcal A(\mathcal S)$ be an $(R,{\bf z})$-cluster algebra with coefficients semifield $\mathbb P$ and initial seed at $t_0$. Then
\begin{eqnarray}
y_{i;t}&=&\prod\limits_{j=1}^ny_{j;t_0}^{c_{ji}^t}\prod\limits_{j=1}^n\left(F_{j;t}|_{\mathbb P}({\bf y}_{t_0},{\bf z})\right)^{b_{ji}^t},\nonumber\\
x_{i;t}&=&\left(\prod\limits_{j=1}^nx_{j;t_0}^{g_{ji}^t}\right)\frac{F_{i;t}|_{\mathcal F}(\hat {\bf y}_{t_0},{\bf z})}{F_{i;t}|_{\mathbb P}({\bf y}_{t_0},{\bf z})}.\nonumber
\end{eqnarray}
\end{theorem}

\section{Main results}

In this section, we give our main results.

\begin{lemma}\cite[Lemma 4.20]{CL}\label{lemCL}
Let $\mathcal A(\mathcal S)$ be an $(R,{\bf z})$-cluster algebra with principal coefficients at $t_0$,  and $D_t^{t_0}=({\bf d}_{1;t}^{t_0},\cdots,{\bf d}_{n;t}^{t_0})$ be the $D$-matrix of ${\bf x}_t$  with respect to ${\bf x}_{t_0}$. If there exists a permutation $\sigma$ of $\{1,\cdots,n\}$ such that
${\bf d}_{j;t}^{t_0}={\bf d}_{\sigma(j);t_0}^{t_0}$ for $j=1,\cdots,n$, then $x_{j;t}=x_{\sigma(j);t_0}$ holds for $j=1,\cdots,n$.
\end{lemma}

\begin{proposition}\label{prokey}
Let $\mathcal A(\mathcal S)$ be an $(R,{\bf z})$-cluster algebra,  $(B_{t_0}$,${\bf x}_{t_0},{\bf y}_{t_0})$ and $(B_t,{\bf x}_t,{\bf y}_t)$ be two seeds of $\mathcal A(\mathcal S)$. Let $D_t^{t_0}=({\bf d}_{1;t}^{t_0},\cdots,{\bf d}_{n;t}^{t_0})$ be the $D$-matrix of ${\bf x}_t$ with respect to ${\bf x}_{t_0}$, and $S$ be a skew-symmetrizer of $RB_{t_0}$. If there exists a permutation $\sigma$ of $\{1,\cdots,n\}$ such that ${\bf d}_{j;t}^{t_0}={\bf d}_{\sigma(j);t_0}^{t_0}$ for any $j=1,\cdots,n$.  Then for each $k\in\{1,\cdots,n\}$, we have

(i) $r_k=r_{\sigma(k)},\;s_k=s_{\sigma(k)}$ and  $z_{k,s}=z_{\sigma(k),s}$, where $s=1,\cdots,r_k-1$. In particular, the mutation polynomials $Z_k$ and $Z_{\sigma(k)}$ are equal;

(ii) $x_{k;t}=x_{\sigma(k);t_0},\; y_{k;t}=y_{\sigma(k);t_0}$ and $b_{ik}^t=b_{\sigma(i)\sigma(k)}^{t_0}$,
where $i=1,\cdots,n$.
\end{proposition}
\begin{proof}
Let $\mathcal A(\mathcal S^{pr})$ be an $(R,{\bf z})$-cluster algebra with principal coefficients at the seed $(B_{t_0}^{pr},{\bf x}_{t_0}^{pr},{\bf y}_{t_0}^{pr})$ satisfying $B_{t_0}^{pr}=B_{t_0}$. Let $(D_t^{t_0})^{pr}$ be the $D$-matrix of ${\bf x}_t^{pr}$ with respect to ${\bf x}_{t_0}^{pr}$.
By Corollary \ref{cordmat},  we know that $(D_t^{t_0})^{pr}=D_t^{t_0}$.
Since ${\bf d}_{j;t}^{t_0}={\bf d}_{\sigma(j);t_0}^{t_0}$ holds for any $j=1,\cdots,n$, we know that $({\bf d}_{j;t}^{t_0})^{pr}=({\bf d}_{\sigma(j);t_0}^{t_0})^{pr}$ holds for $j=1,\cdots,n$.  If we can prove the results in (i), (ii) hold for $\mathcal A(\mathcal S^{pr})$, then by Theorem \ref{thmnt}, we can get that they also hold for $\mathcal A(\mathcal S)$. Thanks to this, we can safely assume that  $\mathcal A(\mathcal S)$ itself is an $(R,{\bf z})$-cluster algebra with principal coefficients at $t_0$.

By Lemma \ref{lemCL}, we get that $$x_{j;t}=x_{\sigma(j);t_0}$$ holds for $j=1,\cdots,n$. So the $G$-matrix and the $H$-matrix of ${\bf x}_t$ are given by
$$G_t=({\bf e}_{\sigma(1)},\cdots,{\bf e}_{\sigma (n)})=H_{t_0}^t,$$ where ${\bf e}_i$ is the $i$-th column vector of $I_n$.

By Proposition \ref{procg}, we can get the $C$-matrix of $(B_t,{\bf x}_t,{\bf y}_t)$ is given by
\begin{eqnarray}\label{eqncmat}
\hspace{14mm}C_t=R^{-1}S^{-1}(G_t^{\rm T})^{-1}SR=(c_{ij}^t),\;\;\text{where }\;c_{ij}^t=\begin{cases}\frac{r_j}{r_{\sigma(j)}}\cdot \frac{s_j}{s_{\sigma(j)}},&i=\sigma(j);\\0,&i\neq \sigma(j).\end{cases}
\end{eqnarray}

By the cluster formula, we know that $H_{t_0}^t(B_tR^{-1}S^{-1})(H_{t_0}^t)^{\rm T}=B_{t_0}R^{-1}S^{-1}$. By comparing the $(\sigma(i),\sigma(k))$-entry of both sides, we get  $b_{ik}^tr_k^{-1}s_k^{-1}=b_{\sigma(i)\sigma(k)}^{t_0}r_{\sigma(k)}^{-1}s_{\sigma(k)}^{-1}$, i.e., we have
\begin{eqnarray}\label{eqn0}
b_{ik}^t=b_{\sigma(i)\sigma(k)}^{t_0}\cdot \frac{r_k}{r_{\sigma(k)}}\cdot \frac{s_k}{s_{\sigma(k)}}.
\end{eqnarray}

We write ${\bf x}_{t}=(x_1,\cdots,x_n)$, then we know that $x_{\sigma(j);t_0}=x_{j;t}=x_j$, where $j=1,\cdots,n$.
Now we fix a $k\in \{1,\cdots,n\}$, let $t^{~\underline{\quad k \quad}}~ t^{\prime}$ and $t_0^{~\underline{\quad \sigma(k) \quad}}~ t_1$ be the subgraph of $\mathbb T_n$.
 By the definition of $(R,{\bf z})$-seed mutation, we have the following equalities.
\begin{eqnarray}
\label{eqn1}
 x_{k;t^\prime}x_{k;t}&=&\left(\prod\limits_{i=1}^nx_{i;t}^{[-b_{ik}^t]_+}\right)^{r_k}\frac{Z_k(\hat y_{k;t})}{Z_k|_{\mathbb P}(y_{k;t})};\\
\label{eqn2} x_{\sigma(k);t_1}x_{\sigma(k);t_0}&=&\left(\prod\limits_{i=1}^nx_{i;t_0}^{[-b_{i\sigma(k)}^{t_0}]_+}\right)^{r_{\sigma(k)}}\frac{Z_{\sigma(k)}(\hat y_{\sigma(k);t_0})}{Z_{\sigma(k)}|_{\mathbb P}(y_{\sigma(k);t_0})},
\end{eqnarray}
where $Z_k$ and $Z_{\sigma(k)}$ are the corresponding mutation polynomials. Denote by $$U_{k;t}=y_{k;t}\prod\limits_{i=1}^nx_{i;t}^{[b_{ik}^t]_+},\;V_{k;t}=\prod\limits_{i=1}^nx_{i;t}^{[-b_{ik}^t]_+},$$ then we know that
\begin{eqnarray}
P&:=&\prod\limits_{i=1}^n(x_{i;t}^{[-b_{ik}^t]_+})^{r_k}Z_k(\hat y_{k;t})=V_{k;t}^{r_k}+z_{k,1}V_{k;t}^{r_k-1}U_{k;t}+\cdots+z_{k;r_k-1}V_{k;t}U_{k;t}^{r_k-1}+U_{k;t}^{r_k},\nonumber\\
Q&:=&\prod\limits_{i=1}^n(x_{i;t_0}^{[-b_{i\sigma(k)}^{t_0}]_+})^{r_{\sigma(k)}}Z_{\sigma(k)}(\hat y_{\sigma(k);t_0})=V_{\sigma(k);t_0}^{r_{\sigma(k)}}+z_{\sigma(k),1}V_{\sigma(k);t_0}^{r_{\sigma(k)}-1}U_{\sigma(k);t_0}+\cdots+
U_{\sigma(k);t_0}^{r_{\sigma(k)}}.\nonumber
\end{eqnarray}
Note that both $P$ and $Q$ are polynomials in $\mathbb {ZP}[x_1,\cdots,x_{k-1},x_{k+1},\cdots,x_n]$.
Now the equalities (\ref{eqn1}) and (\ref{eqn2}) can be expressed as follows:
\begin{eqnarray}
 x_{k;t^\prime}x_k&=&\frac{P}{Z_k|_{\mathbb P}(y_{k;t})};\nonumber\\
 x_{\sigma(k);t_1}x_k&=&\frac{Q}{Z_{\sigma(k)}|_{\mathbb P}(y_{\sigma(k);t_0})},\nonumber
\end{eqnarray}
We can get that
\begin{eqnarray}
\label{eqn3}
x_{k;t^\prime}&=&\frac{Z_{\sigma(k)}|_{\mathbb P}(y_{\sigma(k);t_0})}{Z_k|_{\mathbb P}(y_{k;t})}\cdot\frac{P}{Q}\cdot x_{\sigma(k);t_1},\\
\label{eqn4}
x_{\sigma(k);t_1}&=&\frac{Z_k|_{\mathbb P}(y_{k;t})}{Z_{\sigma(k)}|_{\mathbb P}(y_{\sigma(k);t_0})}\cdot\frac{Q}{P}\cdot x_{k;t^\prime}.
\end{eqnarray}
The equality (\ref{eqn3}) is the expansion of $x_{k;t^\prime}$ with respect to ${\bf x}_{t_1}$ and the equality (\ref{eqn4}) is the expansion of
$x_{\sigma(k);t_1}$ with respect to ${\bf x}_{t^\prime}$. By the Laurent phenomenon, we can get both $\frac{P}{Q}$ and $\frac{Q}{P}$ are Laurent polynomial in $\mathbb{ZP}[x_1,\cdots,x_{k-1},x_{k+1},\cdots,x_n]$. Thus we get $\frac{P}{Q}$ is a Laurent monomial in $\mathbb{ZP}[x_1,\cdots,x_{k-1},x_{k+1},\cdots,x_n]$.
Since both $P$ and $Q$ can not be divided by any $x_j$. We can get $\frac{P}{Q}=1$, i.e., we have $$P=Q.$$

In the following proof, we divide two cases. Case (a): the $k$-th column vector of $B_t$ is a zero vector; Case (b): the $k$-th column vector of $B_t$ is  a nonzero vector.

Case (a). In this case, the $k$-th column vector of $B_t$ is a zero vector. So the rank $1$ generalized cluster algebra generated by $x_{k;t}=x_k$ would split off from $\mathcal A(\mathcal S)$. In this case, $x_k=x_{\sigma(k)}$ actually implies that $k=\sigma(k)$. So $r_k=r_{\sigma(k)}, s_k=s_{\sigma(k)}$ and $z_{k,s}=z_{\sigma(k),s}$ hold. By the equality (\ref{eqncmat}), we know that the $k$-th column vector of $C_t$ is ${\bf e}_{\sigma(k)}={\bf e}_{k}$. So ${ y}_{k;t}=y_k=y_{k;t_0}=y_{\sigma(k);t_0}$.
Since the seed at $t_0$  is obtained from the seed at $t$ by a sequence of $(R,{\bf z})$-seed mutations, we get the $\sigma(k)=k$-th column vector of $B_{t_0}$ is also a zero vector. In particular, $b_{ik}^t=b_{\sigma(i)\sigma(k)}^{t_0}$ holds for $i=1,\cdots,n$.

Case (b). In this case, the $k$-th column vector of $B_t$ is a nonzero vector. So there exists $i_0$ such that $b_{i_0k}^t\neq 0$. Without loss of generality, we can assume that $b_{i_0k}^t>0$. By the equality (\ref{eqn0}), we also have $b_{\sigma(i_0)\sigma(k)}^{t_0}>0$. We just view $P$ and $Q$ as polynomials in $x_{i_0;t}=x_{i_0}$. We know that $P$ is a sum of $r_k+1$ distinct monomials, while $Q$ is a sum of $r_{\sigma(k)}+1$ distinct monomials. By $P=Q$, we can get $$r_k=r_{\sigma(k)}.$$
The highest exponent of $x_{i_0;t}=x_{i_0}$ in $P$ is $b_{i_0k}^tr_k$, while the highest exponent of $x_{\sigma(i_0);t_0}=x_{i_0;t}=x_{i_0}$ in $Q$
is $b_{\sigma(i_0)\sigma(k)}^{t_0}r_{\sigma(k)}$. By $P=Q$, we get $b_{i_0k}^tr_k=b_{\sigma(i_0)\sigma(k)}^{t_0}r_{\sigma(k)}$, i.e., we have
$$b_{i_0k}^t=b_{\sigma(i_0)\sigma(k)}^{t_0}\cdot\frac{r_{\sigma(k)}}{r_k}=b_{\sigma(i_0)\sigma(k)}^{t_0}>0.$$
By the equality (\ref{eqn0}), we know that $$b_{i_0k}^t=b_{\sigma(i_0)\sigma(k)}^{t_0}\cdot \frac{r_k}{r_{\sigma(k)}}\cdot \frac{s_k}{s_{\sigma(k)}}=b_{\sigma(i_0)\sigma(k)}^{t_0}\cdot \frac{s_k}{s_{\sigma(k)}}.$$
By comparing the above two equalities, we get $\frac{s_{\sigma(k)}}{s_k}=1$, i.e., $s_k=s_{\sigma(k)}$.
By $r_k=r_{\sigma(k)},\;s_k=s_{\sigma(k)}$ and the equality (\ref{eqn0}), we get that $b_{ik}^t=b_{\sigma(i)\sigma(k)}^{t_0}$ holds for $i=1,\cdots,n$.
By the equality (\ref{eqncmat}), we know that the $k$-th column vector is ${\bf e}_{\sigma(k)}={\bf e}_{k}$. So ${ y}_{k;t}=y_k=y_{k;t_0}=y_{\sigma(k);t_0}$.
By $r_k=r_{\sigma(k)},\;y_{k;t}=y_{\sigma(k);t_0},\;b_{ik}^t=b_{\sigma(i)\sigma(k)}^{t_0}$ for any $i$, and by comparing the coefficients before the monomials in $P$ and $Q$, we can get $z_{k,s}=z_{\sigma(k),s}$, where $s=1,\cdots,r_k-1$. This completes the proof.
\end{proof}

\begin{theorem}\cite[Theorem 6.3]{CL1}\label{thmdposi} Let  $\mathcal A(\mathcal S)$ be an $(I_n,\phi)$-cluster algebra with initial seed at $t_0$, and ${\bf d}_{i;t}^{t_0}= (d_1,\cdots,d_n)^{\rm T}$ be the $d$-vector of the cluster variable $x_{i;t}$ with respect to the cluster ${\bf x}_{t_0}$ of $\mathcal A(\mathcal S)$. Then for each $k\in\{1,\cdots,n\}$,

(i) $d_k$ depends only on $x_{i;t}$ and $x_{k;t_0}$, not on the clusters containing $x_{k;t_0}$;

(ii)  $d_k\geq -1$ for $k=1,\cdots,n$, and in details,
\begin{eqnarray}
d_k=\begin{cases} -1~,& \text{iff}\;\; x_{i,t}=x_{k,t_0};\\ 0~,& \text{iff}\;\;x_{i,t}\not=x_{k,t_0}\;\;\text{and}\;\; x_{i,t}, x_{k,t_0}\in {\bf x}_{t^\prime} \; \text{for some}\; t^\prime;\\ \text{a positive integer}~, & \text{iff}\;\; \text{there exists no cluster } {\bf x}_{t^\prime} \text{ containing both }x_{i,t} \text{ and }x_{k,t_0} .\end{cases}\nonumber
\end{eqnarray}
In particular, if $x_{i;t}\notin{\bf x}_{t_0}$, then ${\bf d}_{i;t}^{t_0}$ is a nonnegative vector.
\end{theorem}

\begin{theorem}\cite{CL1}\label{thmconnect}
Let $\mathcal A(\mathcal S)$ be an $(I_n,\phi)$-cluster algebra (i.e. classic cluster algebra), then the seeds
whose clusters contain particular cluster variables form a connected subgraph of the exchange graph
of $\mathcal A(\mathcal S)$.
\end{theorem}
\begin{remark}
Note that the statements in Theorem \ref{thmdposi} and Theorem \ref{thmconnect} come from  \cite[Conjecture 7.4]{FZ3} and \cite[Conjecture 4.14(3)]{FZ2} respectively.
\end{remark}
Let $I$ be a subset of $\{1,\cdots,n\}$. We say that $(k_1,\cdots,k_s)$ is an {\bf $I$-sequence}, if $k_j\in I$ for $j=1,\cdots,s$.

\begin{theorem}\label{thmgraph}
Let $\mathcal A(\mathcal S)$ be an $(R,{\bf z})$-cluster algebra, then the seeds
whose clusters contain particular cluster variables form a connected subgraph of the exchange graph ${\bf EG}(\mathcal A(\mathcal S))$
of $\mathcal A(\mathcal S)$.
\end{theorem}
\begin{proof}
For any  $t_0\in \mathbb T_n$  and any  subset $J$ of the cluster variables in ${\bf x}_{t_0}$, we consider the seeds of $\mathcal A(\mathcal S)$
whose clusters contain $J$. We need to check that if these seeds form a connected
subgraph of ${\bf EG}(\mathcal A(\mathcal S))$. Without loss of generality, we can assume that $J=\{x_{p+1;t_0},\cdots,x_{n;t_0}\}$.

Let $(B_{t},{\bf x}_{t},{\bf y}_{t})$ be a seed of $\mathcal A(\mathcal S)$ and the cluster ${\bf x}_{t}$ contains the cluster variables $x_{p+1;t_0},\cdots,x_{n;t_0}$. It suffices to find a vertex $u\in\mathbb T_n$ such that $u$ is connected with $t_0$ by a $\{1,\cdots,p\}$-sequence on $\mathbb T_n$ and the seed $(B_u,{\bf x}_u,{\bf y}_u)$ is equivalent to the seed $(B_{t},{\bf x}_{t},{\bf y}_{t})$.

Let $\mathcal A(\mathcal S^\circ)$ be an $(I_n,\phi)$-cluster algebra (i.e. classic cluster algebra) with initial seed $(B_{t_0}^\circ,{\bf x}_{t_0}^\circ,{\bf y}_{t_0}^\circ)$, where $B_{t_0}^\circ=B_{t_0}R$.
For any $v,w\in\mathbb T_n$, let $D_v^{w}=({\bf d}_{1;v}^{w},\cdots,{\bf d}_{n;v}^{w})$ be the $D$-matrix of ${\bf x}_v$ with respect to ${\bf x}_{w}$, and $(D_v^{w})^\circ$ be the $D$-matrix of ${\bf x}_v^\circ$ with respect to ${\bf x}_{w}^\circ$. By  Corollary \ref{cordmat}, we know that $$D_v^{w}=(D_v^{w})^\circ.$$

Since the cluster ${\bf x}_{t}$ contains the cluster variables $x_{p+1;t_0},\cdots,x_{n;t_0}$, we know that the $D$-matrix $D_t^{t_0}=(D_t^{t_0})^\circ$ contains the $d$-vectors ${\bf d}_{p+1;t_0}^{t_0},\cdots,{\bf d}_{n;t_0}^{t_0}$. Note that ${\bf d}_{i;t_0}^{t_0}=-{\bf e}_i$ which is a non-positive vector. Then by Theorem \ref{thmdposi}, we know that the cluster ${\bf x}_t^\circ$ contains the cluster variables $x_{p+1;t_0}^\circ,\cdots,x_{n;t_0}^\circ$. Then by Theorem \ref{thmconnect}, there exists a vertex $u$ of $\mathbb T_n$ satisfying that

(a) $u$ is connected with $t_0$ by an $\{1,\cdots,p\}$-sequence on $\mathbb T_n$, i.e., we have
$$t_0^{~\underline{  \quad k_1\quad   }}~ t_1^{~\underline{\quad k_2 \quad}} ~t_2^{~\underline{\quad k_3 \quad}} ~\cdots ~t_{s-1} ^{~\underline{~\quad k_{s} \quad}}~ t_s=u,$$
where  $k_j\leq p$ for $j=1,\cdots,s$;

(b) the two seeds $(B_u^\circ,{\bf x}_u^\circ,{\bf y}_u^\circ)$ and $(B_{t}^\circ,{\bf x}_{t}^\circ,{\bf y}_{t}^\circ)$ are equivalent via a permutation $\sigma$. In particular, $x_{j;t}^\circ=x_{\sigma(j);u}^\circ$ for $j=1,\cdots,n$.

By (b), we know that $({\bf d}_{j;t}^u)^\circ=({\bf d}_{\sigma(j);u}^u)^\circ$ for $j=1,\cdots,n$. By $D_t^u=(D_t^u)^\circ$, we get that
 $${\bf d}_{j;t}^u={\bf d}_{\sigma(j);u}^u, \;\text{where }j=1,\cdots,n.$$
Then by Proposition \ref{prokey}, we get that the seeds $(B_u,{\bf x}_u,{\bf y}_u)$ and $(B_{t},{\bf x}_{t},{\bf y}_{t})$ are equivalent via the permutation $\sigma$.

On the other hand, we know that each seed of $\mathcal A(\mathcal S)$ appearing in the following subgraph
$$t_0^{~\underline{  \quad k_1\quad   }}~ t_1^{~\underline{\quad k_2 \quad}} ~t_2^{~\underline{\quad k_3 \quad}} ~\cdots ~t_{s-1} ^{~\underline{~\quad k_{s} \quad}}~ t_s=u$$
contains the cluster variables $x_{p+1,t_0},\cdots,x_{n;t_0}$, by $k_j\leq p$ for $j=1,\cdots,s$. So the seeds of $\mathcal A(\mathcal S)$
whose clusters contain the cluster variables $x_{p+1,t_0},\cdots,x_{n;t_0}$ form a connected subgraph of the exchange graph ${\bf EG}(\mathcal A(\mathcal S))$
of $\mathcal A(\mathcal S)$.
\end{proof}

\begin{theorem}\label{promain}
Let $\mathcal A(\mathcal S)$ be an  $(R,{\bf z})$-cluster algebra with  initial seed  $(B_{t_0},{\bf x}_{t_0},{\bf y}_{t_0})$, and $\mathcal A(\overline{\mathcal S})$ be an  $(\overline R,\overline{\bf z})$-cluster algebra with  initial seed  $(\overline B_{t_0},\overline{\bf x}_{t_0},\overline{\bf y}_{t_0})$. Let $\mathcal X(\mathcal S)$  be the set of cluster variables of $\mathcal A(\mathcal S)$, and  $\mathcal X(\overline{\mathcal S})$ be the set of cluster variables of $\mathcal A(\overline{\mathcal S})$. If  $B_{t_0}R=\overline B_{t_0}\overline R$,  then the following statements hold.

(i) For any  $t_1,t\in\mathbb T_n$ and any $i_{0}, j_0\in\{1,\cdots,n\}$, $x_{i_0;t_1}=x_{j_0;t}$ if and only if $\overline x_{i_0;t_1}=\overline x_{j_0;t}$, where $x_{i_0;t_1},x_{j_0;t}$ are cluster variables of $\mathcal A(\mathcal S)$ and $\overline x_{i_0;t_1}, \overline x_{j_0;t}$ are the corresponding cluster variables of $\mathcal A(\overline{\mathcal S})$.

(ii) There exists a bijection $\alpha:\mathcal X(\mathcal S)\rightarrow \mathcal X(\overline{\mathcal S})$ given by
$\alpha(x_{i;t})=\overline x_{i;t}$, which induces a bijection from the set of clusters of  $\mathcal A(\mathcal S)$ to the set of clusters of $\mathcal A(\overline{\mathcal S})$.
\end{theorem}
\begin{proof} (i) The proof is similar to that of Theorem \ref{thmgraph}.

For any $v,w\in\mathbb T_n$, let $D_v^{w}=({\bf d}_{1;v}^{w},\cdots,{\bf d}_{n;v}^{w})$ be the $D$-matrix of ${\bf x}_v$ with respect to ${\bf x}_{w}$, and
 $\overline D_v^{w}$  be the $D$-matrix of $\overline {\bf x}_v$ with respect to $\overline {\bf x}_{w}$. By  Corollary \ref{cordmat}, we know that $$D_v^{w}=\overline D_v^{w}.$$

If $x_{i_0;t_1}=x_{j_0;t}$, then by Theorem \ref{thmgraph}, there exists a vertex $u$ of $\mathbb T_n$ satisfying that

(a) $u$ is connected with $t_1$ by an $\{1,\cdots,i_0-1,i_0+1,\cdots,n\}$-sequence on $\mathbb T_n$, i.e., we have
$$t_1^{~\underline{  \quad k_1\quad   }}~ t_2^{~\underline{\quad k_2 \quad}} ~t_3^{~\underline{\quad k_3 \quad}} ~\cdots ~t_{s-1} ^{~\underline{~\quad k_{s-1} \quad}}~ t_s=u,$$
where  $k_j\neq i_0$ for $j=1,\cdots,s-1$;

(b) the two seeds $(B_u,{\bf x}_u,{\bf y}_u)$ and $(B_{t},{\bf x}_{t},{\bf y}_{t})$ are equivalent via a permutation $\sigma$. In particular, $x_{j;t}=x_{\sigma(j);u}$ for $j=1,\cdots,n$.

Since $k_j\neq i_0$ for $j=1,\cdots,s-1$, we can get  $$x_{i_0;t_1}=x_{i_0;u}\;\;\text{and } \;\overline x_{i_0;t_1}=\overline x_{i_0;u}.$$
 By the facts $x_{j_0;t}=x_{i_0;t_1}=x_{i_0;u}$ and $x_{j;t}=x_{\sigma(j);u}$, where $j=1,\cdots,n$, we get that $i_0=\sigma(j_0)$.

By (b), we know that ${\bf d}_{j;t}^u={\bf d}_{\sigma(j);u}^u$ for $j=1,\cdots,n$. Since $D_t^u=\overline D_t^u$, we get
 $$\overline{\bf d}_{j;t}^u=\overline {\bf d}_{\sigma(j);u}^u, \;\text{where }j=1,\cdots,n.$$
Then by Proposition \ref{prokey}, we know that $\overline x_{j_0;t}=\overline x_{\sigma(j_0);u}=\overline x_{i_0;u}$. By $\overline x_{i_0;t_1}=\overline x_{i_0;u}$, we get $\overline x_{j_0;t_1}=\overline x_{i_0;t}$

Similarly, if $\overline x_{j_0;t_1}=\overline x_{i_0;t}$, we can show that $x_{i_0;t_1}=x_{j_0;t}$.

(ii) follows from (i).
\end{proof}

\section{Applications of Theorem \ref{promain}}

In this section, we give two applications of of Theorem \ref{promain}. To be more precisely, we prove some properties of the components of the $d$-vectors in Theorem \ref{thmdvec} and we give a characterization for the clusters of a generalized cluster algebra in Theorem \ref{thmlast}.

\begin{theorem}\label{thmdvec}
Let  $\mathcal A(\mathcal S)$ be an $(R,{\bf z})$-cluster algebra with initial seed at $t_0$, and ${\bf d}_{i;t}^{t_0}= (d_1,\cdots,d_n)^{\rm T}$ be the $d$-vector of the cluster variable $x_{i;t}$ with respect to the cluster ${\bf x}_{t_0}$ of $\mathcal A(\mathcal S)$. Then for each $k\in\{1,\cdots,n\}$,

(i) $d_k$ depends only on $x_{i;t}$ and $x_{k;t_0}$, not on the clusters containing $x_{k;t_0}$;

(ii)  $d_k\geq -1$ for $k=1,\cdots,n$, and in details,
\begin{eqnarray}
d_k=\begin{cases} -1~,& \text{iff}\;\; x_{i,t}=x_{k,t_0};\\ 0~,& \text{iff}\;\;x_{i,t}\not=x_{k,t_0}\;\;\text{and}\;\; x_{i,t}, x_{k,t_0}\in {\bf x}_{t^\prime} \; \text{for some}\; t^\prime;\\ \text{a positive integer}~, & \text{iff}\;\; \text{there exists no cluster } {\bf x}_{t^\prime} \text{ containing both }x_{i,t} \text{ and }x_{k,t_0} .\end{cases}\nonumber
\end{eqnarray}
In particular, if $x_{i;t}\notin{\bf x}_{t_0}$, then ${\bf d}_{i;t}^{t_0}$ is a nonnegative vector.
\end{theorem}

\begin{proof}

Let $\mathcal A(\mathcal S^\circ)$ be an $(I_n,\phi)$-cluster algebra (i.e. classic cluster algebra) with initial seed $(B_{t_0}^\circ,{\bf x}_{t_0}^\circ,{\bf y}_{t_0}^\circ)$, where $B_{t_0}^\circ=B_{t_0}R$. By Corollary \ref{cordmat}, we know that ${\bf d}_{i;t}=({\bf d}_{i;t}^{t_0})^\circ$ is also the $d$-vector of $x_{i;t}^\circ$ with respect to ${\bf x}_{t_0}$.

(i) By Theorem \ref{thmdposi} (i), $d_k$ depends only on $x_{i;t}^\circ$ and $x_{k;t_0}^\circ$, not on the clusters containing $x_{k;t_0}^\circ$. Then by Theorem \ref{promain} (ii), we can get $d_k$ depends only on $x_{i;t}$ and $x_{k;t_0}$, not on the clusters containing $x_{k;t_0}$.

(ii) By Theorem \ref{thmdposi} (ii), we know that  $d_k\geq -1$ for $k=1,\cdots,n$, and in details,
\begin{eqnarray}
d_k=\begin{cases} -1~,& \text{iff}\;\; x_{i,t}^\circ=x_{k,t_0}^\circ;\\ 0~,& \text{iff}\;\;x_{i,t}^\circ\not=x_{k,t_0}^\circ\;\;\text{and}\;\; x_{i,t}^\circ, x_{k,t_0}^\circ\in {\bf x}_{t^\prime}^\circ \; \text{for some}\; t^\prime;\\ \text{a positive integer}~, & \text{iff}\;\; \text{there exists no cluster } {\bf x}_{t^\prime}^\circ \text{ containing both }x_{i,t}^\circ \text{ and }x_{k,t_0}^\circ .\end{cases}\nonumber
\end{eqnarray}
Then by  Theorem \ref{promain} (ii), we can get
 \begin{eqnarray}
d_k=\begin{cases} -1~,& \text{iff}\;\; x_{i,t}=x_{k,t_0};\\ 0~,& \text{iff}\;\;x_{i,t}\not=x_{k,t_0}\;\;\text{and}\;\; x_{i,t}, x_{k,t_0}\in {\bf x}_{t^\prime} \; \text{for some}\; t^\prime;\\ \text{a positive integer}~, & \text{iff}\;\; \text{there exists no cluster } {\bf x}_{t^\prime} \text{ containing both }x_{i,t} \text{ and }x_{k,t_0} .\end{cases}\nonumber
\end{eqnarray}
\end{proof}

Let $\mathcal A(\mathcal S)$ be an $(R,{\bf z})$-cluster algebra, and $\mathcal X(\mathcal S)$ be the set of cluster variables of $\mathcal A(\mathcal S)$. Two cluster variables $x$ and $w$ are said to be {\bf compatible}, if there exists a cluster ${\bf x}_t$ of $\mathcal A(\mathcal S)$ containing both $x$ and $w$. A subset $M\subseteq \mathcal X(\mathcal S)$ is called a {\bf compatible set} of $\mathcal A(\mathcal S)$,  if $x$ and $w$ are compatible for any $x,w\in M$.

The theorem is a conjecture in \cite[Conjecture 5.5]{FST} by Fomin,  Shapiro and  Thurston, which has been proved by the authors in \cite{CL1} for classic cluster algebras.

\begin{lemma}\cite[Theorem 7.4]{CL1}\label{lemcompatible}
Let $\mathcal A(\mathcal S)$ be an $(I_n,\phi)$-cluster algebra (i.e. classic cluster algebra), and $\mathcal X(\mathcal S)$ be the set of cluster variables of $\mathcal A(\mathcal S)$. Then

(i) a subset $M\subseteq \mathcal X(\mathcal S)$ is a compatible set of $\mathcal A(\mathcal S)$ if and only if $M$ is a subset of some cluster (as a set) of $\mathcal A(\mathcal S)$;

(ii) a subset $M\subseteq \mathcal X(\mathcal S)$ is a maximal compatible set of $\mathcal A(\mathcal S)$ if and only if $M$ is a cluster (as a set) of $\mathcal A(\mathcal S)$.
\end{lemma}

\begin{theorem}\label{thmlast}
Let $\mathcal A(\mathcal S)$ be an $(R,{\bf z})$-cluster algebra with initial seed at $t_0$, and $\mathcal X(\mathcal S)$ be the set of cluster variables of $\mathcal A(\mathcal S)$. Then

(i) a subset $M\subseteq \mathcal X(\mathcal S)$ is a compatible set of $\mathcal A(\mathcal S)$ if and only if $M$ is a subset of some cluster (as a set) of $\mathcal A(\mathcal S)$;

(ii) a subset $M\subseteq \mathcal X(\mathcal S)$ is a maximal compatible set of $\mathcal A(\mathcal S)$ if and only if $M$ is a cluster (as a set) of $\mathcal A(\mathcal S)$.
\end{theorem}
\begin{proof}
(i) Let $\mathcal A(\mathcal S^\circ)$ be an $(I_n,\phi)$-cluster algebra (i.e. classic cluster algebra) with initial seed $(B_{t_0}^\circ,{\bf x}_{t_0}^\circ,{\bf y}_{t_0}^\circ)$, where $B_{t_0}^\circ=B_{t_0}R$. Let $\mathcal X(\mathcal S^\circ)$ be the set of cluster variables of $\mathcal A(\mathcal S^\circ)$.
Let $\alpha:\mathcal X(\mathcal S)\rightarrow \mathcal X(\mathcal S^\circ)$ be the bijection given in Theorem \ref{promain} (ii).

By Theorem \ref{promain} (ii), we know that a subset $M\subseteq \mathcal X(\mathcal S)$ is a compatible set of $\mathcal A(\mathcal S)$ if and only if
$M^\circ:=\alpha(M)\subseteq \mathcal X(\mathcal S^\circ)$ is a compatible set of $\mathcal A(\mathcal S^\circ)$. By Lemma \ref{lemcompatible}, $M^\circ$ is a compatible set of $\mathcal A(\mathcal S^\circ)$ if and only if $M^\circ$ is a subset of some cluster ${\bf x}_t^\circ$ (as a set) of $\mathcal A(\mathcal S^\circ)$. By Theorem \ref{promain} (ii), $M^\circ$ is a subset of  ${\bf x}_t^\circ$ (as a set) of $\mathcal A(\mathcal S^\circ)$ if and only if $M$ is a subset of ${\bf x}_t$ (as a set) of $\mathcal A(\mathcal S)$. Hence, we obtain that a subset $M\subseteq \mathcal X(\mathcal S)$ is a compatible set of $\mathcal A(\mathcal S)$ if and only if $M$ is a subset of some cluster (as a set) of $\mathcal A(\mathcal S)$.

(ii) follows from (i).
\end{proof}

{\bf Acknowledgements:}\; The authors are grateful to M. Shapiro and M. Gekhtman for inspiring discussions during the ``Cluster Algebras 2019" (June 3-June 21, 2019) at RIMS, Kyoto University.

\end{document}